\documentclass[leqno,11pt,a4paper]{amsart}
\usepackage{bw}
\newcommand\blfootnote[1]{%
  \begingroup
  \renewcommand\thefootnote{}\footnote{#1}%
  \addtocounter{footnote}{-1}%
  \endgroup
}
\begin{document}
\title[Wall-crossing for iterated Hilbert schemes]{Wall-crossing for iterated Hilbert schemes (or `Hilb of Hilb')}
\vs
\maketitle

\begin{abstract} We study wall-crossing phenomena in the McKay correspondence. Craw--Ishii show that every projective crepant resolution of a Gorenstein abelian quotient singularity arises as a moduli space of $\theta$-stable representations of the McKay quiver. The stability condition $\theta$ moves in a vector space with a chamber decomposition in which (some) wall-crossings capture flops between different crepant resolutions. We investigate where chambers for certain resolutions with Hilbert scheme-like moduli interpretations -- iterated Hilbert schemes, or `Hilb of Hilb' -- sit relative to the principal chamber defining the usual $G$-Hilbert scheme. We survey relevant aspects of wall-crossing, pose our main conjecture, prove it for some examples and special cases, and discuss connections to other parts of the McKay correspondence.
\end{abstract}

\section{Introduction}

\subsection{Overview}

The McKay correspondence\blfootnote{MSC 2020: 14E16 (primary), 14M25, 14J17 (secondary)} studies minimal or crepant resolutions of Gorenstein quotient singularities $\C^n/G$ for $G\subseteq\on{SL}_n(\C)$ a finite subgroup. In dimensions $2$ and $3$, where there is the guarantee of at least one crepant resolution, there are established connections between the geometry of such resolutions and representation theory (either directly of the group $G$, or of related objects) \cite{av_ref_85,bd_stro_96,in_mck_00,mck_car_81,rei_cor_02}. We will subsequently focus on dimension $3$.

The first crepant resolution of $\C^3/G$ that was studied in detail is the \emph{$G$-Hilbert scheme} $G\hilb\C^3$. This space is the moduli space of $G$-clusters -- $0$-dimensional, $G$-invariant subschemes $Z\subseteq\C^3$ with $H^0(\mO_Z)\cong\C[G]$ as $G$-modules -- and was shown to be smooth for all $G\subseteq\on{SL}_3(\C)$ in \cite{bkr_mck_01} after the abelian case was studied explicitly in \cite{nak_hil_01}.

Quivers and their representation theory have been known to enter the picture for some time \cite{cra_mck_01,bk_mck_04,dec_dih_12}. One notable instance is \cite[Thm.~1.1]{ci_flo_04} where it is shown that for abelian $G$ every crepant resolution of $\C^3/G$ can be realised as a moduli space of quiver representations. The quiver in this situation is the \emph{McKay quiver} $Q_G$, which is built out of the representation theory of $G$. Namely, $Q_G$ has a vertex for each irreducible representation of $G$ and the number of arrows between two vertices $\rho$ and $\rho'$ is given by
$$\on{dim}\on{Hom}_G(\rho',\rho\otimes\C^3)$$
which is the multiplicity of $\rho'$ in the decomposition of $\rho\otimes\C^3$ into irreducible representations.

The focus of this paper is in developing an understanding of how different crepant resolutions of $\C^3/G$ are related, both geometrically by flops and more delicately by GIT wall-crossing as in \cite{ci_flo_04}, with especial focus on the \emph{iterated Hilbert schemes} (or \emph{`Hilb of Hilb'}) studied in \cite{iin_gnh_13}. We will describe our approach in \S\ref{sec:wall}-\ref{sec:it_hilb}, state our main conjecture that epitomises it in \S\ref{sec:conj}, and work out this conjecture in some examples and special cases in \S\ref{sec:ex}.

\subsection{Toric geometry}

When $G$ is abelian the singularity $\C^3/G$ and its crepant resolutions are toric $3$-folds, enabling one to use combinatorial methods to examine them. We briefly recall the setup for toric geometry and fix notation that we will use throughout the paper. We will write $G=\frac{1}{r}(a,b,c)$ to mean that $G\cong\Z/r$ is generated by
$$\mat{ccc}
\eps^a \\
& \eps^b \\
& & \eps^c\tam$$
where $\eps$ is a primitive $r$th root of unity. We will assume $a+b+c\equiv0\on{mod}{r}$; in other words, that $G\subseteq\on{SL}_3(\C)$.

Let $N=\Z^3$ and $N'=\Z^3+\Z\cdot(\tfrac{a}{r},\tfrac{b}{r},\tfrac{c}{r})$. Let $\sigma$ (resp.~$\sigma'$) denote the cone in $N_\R$ (resp.~in $N'_\R$) generated by the standard basis vectors $e_1=(1,0,0),e_2=(0,1,0),e_3=(0,0,1)$. The toric variety $U_\sigma$ associated to $\sigma$ is $\C^3$, and to $\sigma'$ is $\C^3/G$ with the inclusion of lattices $N\subseteq N'$ inducing the quotient map $\C^3\to\C^3/G$. We refer to \cite{cls_tor_11} for general information on toric varieties. Denote by $\Delta_1$ the slice $\sigma'\cap(e^1+e^2+e^3=1)$, where $e^i$ are the dual basis of $N^\vee_\R$ to $e_i$. We call $\Delta_1$ the \emph{junior simplex}. With this setup a crepant resolution of $\C^3/G$ corresponds to a triangulation of $\Delta_1$ with vertices in $N'$ such that each triangle is unimodal, meaning that its vertices form a $\Z$-basis of $N'$.

Let $\pi\colon Y\to\C^3/G$ be a crepant resolution and let $\mathcal{T}$ be the corresponding triangulation. We have the following correspondences between the geometry of $Y$ and the combinatorics of $\mathcal{T}$:
\begin{itemize}
\item torus-invariant exceptional curves in $Y\longleftrightarrow$ edges in $\mathcal{T}$
\item torus-invariant exceptional divisors in $Y\longleftrightarrow$ vertices in $\mathcal{T}$
\item torus-invariant compact exceptional divisors in $Y\longleftrightarrow$ interior vertices in $\mathcal{T}$
\item torus-fixed points in $Y\longleftrightarrow$ triangles in $\mathcal{T}$
\end{itemize}

The combinatorial avatar of the flop in a torus-invariant exceptional curve $C$ corresponds to \emph{flipping} the edge corresponding to $C$ in $\mathcal{T}$ as shown in Fig.~\ref{fig:flip}.

\begin{figure}[h]
\begin{center}
\begin{tikzpicture}[scale=1.4]
\foreach \i in {1,...,2}
{
\node (a\i) at (0,\i){\tiny$\bullet$};
}
\foreach \i in {1,...,2}
{
\node (b\i) at (1,\i+0.5){\tiny$\bullet$};
}

\draw[line width = 1.2pt](a1.center) to (b2.center);
\draw (a1.center) to (a2.center);
\draw (a1.center) to (b1.center);
\draw (a2.center) to (b2.center);
\draw (b1.center) to (b2.center);

\foreach \i in {1,...,2}
{
\node (a\i) at (2,\i){\tiny$\bullet$};
}
\foreach \i in {1,...,2}
{
\node (b\i) at (3,\i+0.5){\tiny$\bullet$};
}

\draw[line width = 1.2pt](a2.center) to (b1.center);
\draw (a1.center) to (a2.center);
\draw (a1.center) to (b1.center);
\draw (a2.center) to (b2.center);
\draw (b1.center) to (b2.center);
\end{tikzpicture}
\end{center}
\caption{Flipping an edge}
\label{fig:flip}
\end{figure}
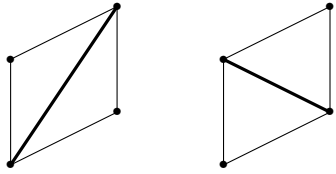

Craw--Reid \cite{cr_how_02} describe an algorithm for computing the triangulation for $G\hilb\C^3$, which works by dividing the junior simplex into a collection of `regular triangles' \cite[\S1.2]{cr_how_02} and then applying a standard subdivision to each of these pieces.

\subsection{Quiver representations}

Let $Q=(Q_0,Q_1)$ be a quiver, with vertex set $Q_0$ and arrow set $Q_1$. For an arrow $\alpha\in Q_1$ we denote by $h(\alpha)$ and $t(\alpha)$ its head and its tail respectively. A representation of $Q$ with dimension vector $\ub{d}=(d_i)_{i\in Q_0}$ is an assignment of a $d_i$-dimensional complex vector space $V_i$ to each vertex $i$, and a linear map
$$f_\alpha\colon V_{t(\alpha)}\to V_{h(\alpha)}$$
to each arrow $\alpha\in Q_1$. If $V$ is a representation of $Q$, we write $\dimv(V)=(\on{dim}{V_i})_{i\in Q_0}$. Choose a `stability condition'
$$\theta\in\Theta(Q,\ub{d}):=\{\vartheta\in\on{Hom}_\Z(\Z^{Q_0},\Q):\vartheta(\ub{d})=0\}$$
We say that a representation $V$ of $Q$ is $\theta$-stable if every nontrivial proper subrepresentation $U\subsetneq V$ has $\theta(\dimv(U))>0$. It is $\theta$-semistable if the strict inequality is weakened to $\geq$.
This data produces a moduli space $\M_\theta(Q,\ub{d})$ parameterising $\theta$-(semi)stable representations $V$ of $Q$ with $\dimv(V)=\ub{d}$. We will in fact usually work with representations in which the linear maps $f_\alpha$ satisfy certain relations coming from a superpotential \cite{bsw_sup_10}. The setup above is unmodified for this situation, and so we will not spell it out here.

It turns out that for generically chosen $\theta$ the moduli space does not depend on small perturbations of $\theta$ in $\Theta(Q,\ub{d})$. Thus there is a chamber decomposition
$$\Theta(Q,\ub{d})=\bigcup\ol{\mfk{C}}$$
where each open chamber $\mfk{C}$ has the property that $\theta,\vartheta\in\mfk{C}$ implies $\M_\theta(Q,\ub{d})\cong\M_\vartheta(Q,\ub{d})$. We denote this space by $\M_\mfk{C}$. Note that there may be many chambers describing isomorphic moduli spaces. In the particular situation of the McKay quiver $Q_G$ equipped with a natural dimension vector and certain `preprojective' relations, the moduli spaces $\M_\theta(Q_G,\ub{d})$ are crepant resolutions of $\C^3/G$.

\begin{thm}[\!\!\!{\cite[Thm.~1.2]{bkr_mck_01} +\cite[Thm.~1.1]{ci_flo_04}}] Let $G\subseteq\on{SL}_3(\C)$ be a finite subgroup. Set $\ub{d}=(\on{dim}{\rho})_{\rho\in\on{Irr}(G)}$. For each chamber $\mfk{C}\subseteq\Theta(Q_G,\ub{d})$ the moduli space $\M_\mfk{C}$ is a crepant resolution of $\C^3/G$. When $G$ is abelian, every projective crepant resolution of $\C^3/G$ arises in this way.
\end{thm}

Since the vertex set of $Q_G$ is identified with the set $\on{Irr}(G)$ of irreducible representations of $G$, we can regard stability conditions $\theta$ as maps from the representation ring of $G$ to $\Q$. In this version the condition $\theta(\ub{d})=0$ translates to $\theta$ evaluating to zero on the regular representation. Stability conditions $\theta$ that evaluate positively on nontrivial representations give rise to moduli spaces isomorphic to $G\hilb\C^3$. We denote the set of such stability conditions by $\Theta^+(Q_G,\ub{d})$ and the chamber containing them by $\mfk{C}_0$.

\section{Wall-crossing} \label{sec:wall}

As one passes from a chamber $\mfk{C}\subseteq\Theta(Q_G,\ub{d})$ to another chamber $\mfk{C}'$ through the `wall' $\ol{\mfk{C}}\cap\ol{\mfk{C}'}$ there is a birational map $\M_\mfk{C}\to\M_{\mfk{C}'}$ obtained by factoring through the moduli space of semistable representations corresponding to a generic stability condition on the wall.

Let $\pi_\mfk{C}$ denote the morphism $\M_\mfk{C}\to\C^3/G$. Often some exceptional locus $E\subseteq\pi_\mfk{C}^{-1}(0)$ is contracted in passing to the moduli of semistable representations. This phenomenon has been studied in \cite{ci_flo_04} and in depth for the chamber $\mfk{C}_0$ giving $G\hilb\C^3$ in \cite{wor_wal_20} when $G$ is abelian. One of the key tools in the latter is \emph{Reid's recipe} \cite{rei_mck_97,cra_exp_05}, which labels strata of $\pi_{\mfk{C}_0}^{-1}(0)$ with irreducible representations of $G$.

\begin{example} Consider $G=\frac{1}{6}(1,2,3)$. Reid's recipe for this group is shown in Fig.~\ref{fig:1/6a}. An integer $a$ denotes the representation $\rho_a\colon G\to\C^\times,g\mapsto\eps^a$.

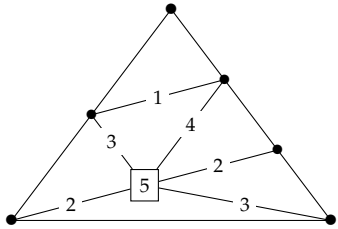
\begin{figure}[h]
\begin{center}
\begin{tikzpicture}[scale=0.7]

\node(e1) at (0+8,2){$\bullet$};
\node(e2) at (3+8,-2){$\bullet$};
\node(e3) at (-3+8,-2){$\bullet$};

\small

\node[draw,fill=white] (123) at (-0.5+8,-4/3){\tiny $5$};
\node(240) at (2+8,2-8/3){$\bullet$};
\node(420) at (1+8,2-4/3){$\bullet$};
\node(303) at (-1.5+8,0){$\bullet$};

{\tiny

\draw[-] (e2.center) to node[fill=white] {$3$} (123);
\draw[-] (e3.center) to node[fill=white] {$2$} (123) to node[fill=white] {$2$} (240.center);
\draw[-] (303.center) to node[fill=white] {$1$} (420.center);
\draw[-] (123) to node[fill=white] {$3$} (303.center);
\draw[-] (123) to node[fill=white] {$4$} (420.center);
}

\draw[-] (e1.center) to (e2.center) to (e3.center) to (e1.center);
\end{tikzpicture}
\end{center}
\caption{Reid's recipe for $G=\frac{1}{6}(1,2,3)$}
\label{fig:1/6a}
\end{figure}
\end{example}

The main result of \cite{wor_wal_20} classifies the walls of $\mfk{C}_0$ when $G$ is abelian and finds the unstable and contracted loci for each wall. We need some more language in order to state it. 

Let $\rho$ be a character marking a curve in $G$-Hilb. We say that the collection or `chain' of curves marked with $\rho$ is a \emph{generalised long side} \cite[Def.~4.12]{wor_wal_20} if it starts and ends on the boundary of the junior simplex, and if all the corresponding edges along the $\chi$-chain are boundary edges of regular triangles. We exclude the fundamentally different case of three lines meeting at a trivalent vertex if there is a meeting of champions \cite[\S2.8.2]{cr_how_02} of side length $0$. We call a curve in a generalised long side \emph{final} \cite[Def.~4.14]{wor_wal_20} if it is the furthest curve along away from a vertex along a straight line segment.

On the left of Figure \ref{fig:final} we show the two generalised long sides for $G=\frac{1}{35}(1,3,31)$ with dashed lines, and on the right we show the corresponding final curves bolded.

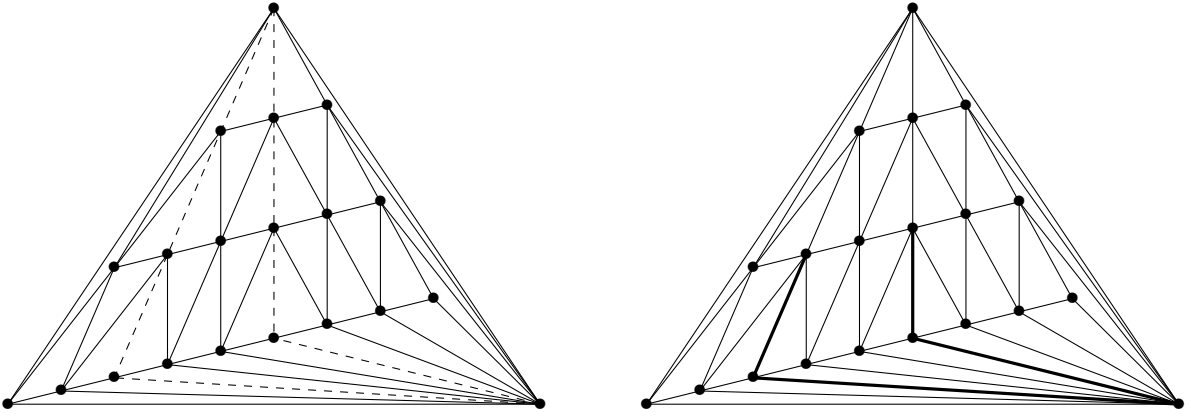
\begin{figure}[h]
\begin{center}
\begin{tikzpicture}[scale=0.35]
\node (e1) at (0,9){$\bullet$};
\node (e2) at (10,-6){$\bullet$};
\node (e3) at (-10,-6){$\bullet$};

\draw (e1.center) to (e2.center) to (e3.center) to (e1.center);

\node (a1) at (-10+2,-6+1/2){$\bullet$};
\node (a2) at (-10+4,-6+2/2){$\bullet$};
\node (a3) at (-10+6,-6+3/2){$\bullet$};
\node (a4) at (-10+8,-6+4/2){$\bullet$};
\node (a5) at (-10+10,-6+5/2){$\bullet$};
\node (a6) at (-10+12,-6+6/2){$\bullet$};
\node (a7) at (-10+14,-6+7/2){$\bullet$};
\node (a8) at (6,-2){$\bullet$};

\node (b1) at (2,9-11/3){$\bullet$};
\node (b2) at (4,9-22/3){$\bullet$};

\node (c1) at (-10+6-2,-6+3/2+11/3){$\bullet$};
\node (c2) at (-10+8-2,-6+4/2+11/3){$\bullet$};
\node (c3) at (-10+10-2,-6+5/2+11/3){$\bullet$};
\node (c4) at (-10+12-2,-6+6/2+11/3){$\bullet$};
\node (c5) at (-10+14-2,-6+7/2+11/3){$\bullet$};

\node (d1) at (2-4,9-2/2-11/3){$\bullet$};
\node (d2) at (2-2,9-1/2-11/3){$\bullet$};

\draw(e2.center) to (a1.center);
\draw[dashed](e2.center) to (a2.center);
\draw(e2.center) to (a3.center);
\draw(e2.center) to (a4.center);
\draw[dashed](e2.center) to (a5.center);
\draw(e2.center) to (a6.center);
\draw(e2.center) to (a7.center);
\draw(e2.center) to (a8.center);

\draw(e2.center) to (b1.center);
\draw(e2.center) to (b2.center);

\draw(e3.center) to (c1.center) to (d1.center);

\draw(e1.center) to (c1.center);
\draw[dashed](e1.center) to (d1.center) to (c2.center);
\draw[dashed] (c2.center) to (a2.center);
\draw[dashed](e1.center) to (d2.center) to (c4.center);
\draw[dashed] (c4.center) to (a5.center);

\draw (e3.center) to (a1.center) to (a2.center) to (a3.center) to (a4.center) to (a5.center) to (a6.center) to (a7.center) to (a8.center);

\draw (e1.center) to (b1.center) to (b2.center) to (a8.center);

\draw (a1.center) to (c1.center);
\draw (a1.center) to (c2.center);
\draw (c1.center) to (c2.center);

\draw (d1.center) to (d2.center);
\draw (a3.center) to (c2.center);
\draw (a4.center) to (c4.center);
\draw (d1.center) to (c3.center) to (a4.center);
\draw (c2.center) to (c3.center) to (c4.center);
\draw (a3.center) to (c3.center) to (d2.center);

\draw (d2.center) to (b1.center);
\draw (a7.center) to (b2.center);
\draw (a6.center) to (c4.center);
\draw (a6.center) to (c5.center) to (b1.center);
\draw (c4.center) to (c5.center) to (b2.center);
\draw (a7.center) to (c5.center) to (d2.center);

\node (e1) at (0+24,9){$\bullet$};
\node (e2) at (10+24,-6){$\bullet$};
\node (e3) at (-10+24,-6){$\bullet$};

\draw (e1.center) to (e2.center) to (e3.center) to (e1.center);

\node (a1) at (-10+2+24,-6+1/2){$\bullet$};
\node (a2) at (-10+4+24,-6+2/2){$\bullet$};
\node (a3) at (-10+6+24,-6+3/2){$\bullet$};
\node (a4) at (-10+8+24,-6+4/2){$\bullet$};
\node (a5) at (-10+10+24,-6+5/2){$\bullet$};
\node (a6) at (-10+12+24,-6+6/2){$\bullet$};
\node (a7) at (-10+14+24,-6+7/2){$\bullet$};
\node (a8) at (6+24,-2){$\bullet$};

\node (b1) at (2+24,9-11/3){$\bullet$};
\node (b2) at (4+24,9-22/3){$\bullet$};

\node (c1) at (-10+6-2+24,-6+3/2+11/3){$\bullet$};
\node (c2) at (-10+8-2+24,-6+4/2+11/3){$\bullet$};
\node (c3) at (-10+10-2+24,-6+5/2+11/3){$\bullet$};
\node (c4) at (-10+12-2+24,-6+6/2+11/3){$\bullet$};
\node (c5) at (-10+14-2+24,-6+7/2+11/3){$\bullet$};

\node (d1) at (2-4+24,9-2/2-11/3){$\bullet$};
\node (d2) at (2-2+24,9-1/2-11/3){$\bullet$};

\draw(e2.center) to (a1.center);
\draw[line width = 1.2pt](e2.center) to (a2.center);
\draw(e2.center) to (a3.center);
\draw(e2.center) to (a4.center);
\draw[line width = 1.2pt](e2.center) to (a5.center);
\draw(e2.center) to (a6.center);
\draw(e2.center) to (a7.center);
\draw(e2.center) to (a8.center);

\draw(e2.center) to (b1.center);
\draw(e2.center) to (b2.center);

\draw(e3.center) to (c1.center) to (d1.center);

\draw(e1.center) to (c1.center);
\draw(e1.center) to (d1.center) to (c2.center);
\draw[line width = 1.2pt] (c2.center) to (a2.center);
\draw(e1.center) to (d2.center) to (c4.center);
\draw[line width = 1.2pt] (c4.center) to (a5.center);

\draw (e3.center) to (a1.center) to (a2.center) to (a3.center) to (a4.center) to (a5.center) to (a6.center) to (a7.center) to (a8.center);

\draw (e1.center) to (b1.center) to (b2.center) to (a8.center);

\draw (a1.center) to (c1.center);
\draw (a1.center) to (c2.center);
\draw (c1.center) to (c2.center);

\draw (d1.center) to (d2.center);
\draw (a3.center) to (c2.center);
\draw (a4.center) to (c4.center);
\draw (d1.center) to (c3.center) to (a4.center);
\draw (c2.center) to (c3.center) to (c4.center);
\draw (a3.center) to (c3.center) to (d2.center);

\draw (d2.center) to (b1.center);
\draw (a7.center) to (b2.center);
\draw (a6.center) to (c4.center);
\draw (a6.center) to (c5.center) to (b1.center);
\draw (c4.center) to (c5.center) to (b2.center);
\draw (a7.center) to (c5.center) to (d2.center);
\end{tikzpicture}
\end{center}
\caption{Generalised long sides and final curves for $G=\frac{1}{35}(1,3,31)$}
\label{fig:final}
\end{figure}

\begin{thm}[\!{\cite[Thm.~4.17]{wor_wal_20}}] \label{thm:worm} Let $G\subseteq\on{SL}_3(\C)$ be abelian. The walls of $\mfk{C}_0$ are as follows:
\begin{enumerate}
\item[\texttt{0}.] one wall for each irreducible exceptional divisor,
\item[\texttt{I}.] one wall for each exceptional $(-1,-1)$-curve,
\item[\texttt{III}.] one wall for each generalised long side,
\item[\texttt{0'}.] the remaining walls come from potentially reducible exceptional divisors.
\end{enumerate}
The unstable locus in each case is:
\begin{enumerate}
\item[\texttt{0}.] the irreducible divisor (not contracted),
\item[\texttt{I}.] the $(-1,-1)$-curve (contracted to a point),
\item[\texttt{III}.] the ruled surface swept out by a final curve of the generalised long side (contracted to a curve),
\item[\texttt{0'}.] the potentially reducible divisor (not contracted).
\end{enumerate}
\end{thm}

We use the enumeration \texttt{0}-\texttt{III} of \cite{wil_kah_92,ci_flo_04} with the slight modification of distinguishing between the two possibilities for type \texttt{0} walls. In this way we can associate at most two representations to each wall of $\mfk{C}_0$.

\begin{definition} Let $G\subseteq\on{SL}_3(\C)$ be abelian. To a wall $\mfk{w}$ of $\mfk{C}_0$ we associate a (possibly empty) set $\chi(\mfk{w})$ of representations of $G$ as follows:
\begin{enumerate}
\item[\texttt{0}.] If $\mfk{w}$ corresponds to an irreducible divisor $D$, let $\chi(\mfk{w})$ be the set of (at most two) representations labelling $D$,
\item[\texttt{I}.] If $\mfk{w}$ corresponds to a $(-1,-1)$-curve $C$, let $\chi(\mfk{w})$ be the singleton consisting of the representation labelling $C$,
\item[\texttt{III}.] If $\mfk{w}$ corresponds to a generalised long side with final curves $C_1$ and $C_2$, let $\chi(\mfk{w})$ be the singleton consisting of the representation labelling $C_1$ and $C_2$,
\item[\texttt{0'}.] Set $\chi(\mfk{w})=\emptyset$ for all other walls.
\end{enumerate}
\end{definition}

Let $\gamma\colon[0,1]\to\Theta(Q_G,\ub{d})$ be a path starting in $\mfk{C}_0$. Suppose $\gamma$ only meets walls generically. Let $\gamma$ pass through walls $\mfk{w}_1,\dots,\mfk{w}_m$ at times $0=t_0<t_1<t_2<\dots<t_m<t_{m+1}=1$. Write $\mfk{C}_i$ for the chamber $\gamma(t)$ lives in for $t_i<t<t_{i+1}$. Notice that this is consistent with our labelling of $\mfk{C}_0$. Write $\M_i=\M_{\mfk{C}_i}$. We suppose further that the $\mfk{C}_i$ are distinct. By \cite[\S8]{ci_flo_04} we may restrict attention to the case that each wall $\mfk{w}_i$ is not of type \texttt{III} from Thm.~\ref{thm:worm}. From \cite[Prop.~6.1]{ci_flo_04} we see that the unstable locus for each such $\mfk{w}_i$ is either a single exceptional $(-1,-1)$-curve or a divisor as in Thm.~\ref{thm:worm}.

We can associate representations to the walls $\mfk{w}_i$ inductively by essentially propagating Reid's recipe to different crepant resolutions in such a way that is `oriented' by $\gamma$. Suppose $\mfk{w}_2$ has unstable locus $E\subseteq\M_1$. This is either:
\begin{enumerate}
\item[\texttt{0}.] an irreducible exceptional divisor $D$,
\item[\texttt{I}.] an exceptional curve $C$,
\item[\texttt{0'}.] a reducible exceptional divisor.
\end{enumerate}

Note that in type \texttt{0} the divisor $D$ corresponds to a vertex in the interior of the junior simplex. We mark $\mfk{w}_i$ with the set of representations marking the divisor in $G\hilb\C^3$ corresponding to that vertex. In type \texttt{I} there are two possibilities: either $\mfk{w}_1$ was type \texttt{I} and $C$ is the curve resulting from the flop, or $C$ is a different curve. If $C$ came from the flop, we let $\chi(\mfk{w}_2)=\chi(\mfk{w}_1)$; if $C$ is a different curve we let $\chi(\mfk{w}_2)$ be the singleton consisting of the representation marking the transform of that $C$ in $G\hilb\C^3$. We continue inductively, thus associating a set of representations $\chi(\mfk{w}_i)$ to each $\mfk{w}_i$. We write
$$\chi(\gamma):=\bigcup_{i=1}^m\chi(\mfk{w}_i)$$

\begin{example} \label{ex:16_wall} We describe this procedure for a path in $\Theta(Q_G,\ub{d})$ for $G=\frac{1}{6}(1,2,3)$. On the left is the triangulation for $G\hilb\C^3$ with the labels from Reid's recipe for two exceptional curves shown. There is a path $\gamma$ starting in $\mfk{C}_0$ and inducing in turn the two flops shown in the centre and rightmost triangulations: first flopping the curve marked with $4$, and then flopping the curve marked with $2$ that is not initially floppable but becomes so after the first flop. In this situation, the path $\gamma$ need only pass through two walls $\mfk{w}_1$ and $\mfk{w}_2$ corresponding to these flops. As discussed -- and as seen explicitly for a group of order $12$ in \cite[Ex.~9.13]{ci_flo_04} -- in general $\gamma$ may need to cross several walls of type \texttt{0} or \texttt{0'} prior to realising flops not directly from $G\hilb\C^3$. The first wall has $\chi(\mfk{w}_1)=\{4\}$, and the second has $\chi(\mfk{w}_2)=\{2\}$ since the curve flopped by crossing $\mfk{w}_2$ was marked with $2$ in $G\hilb\C^3$.

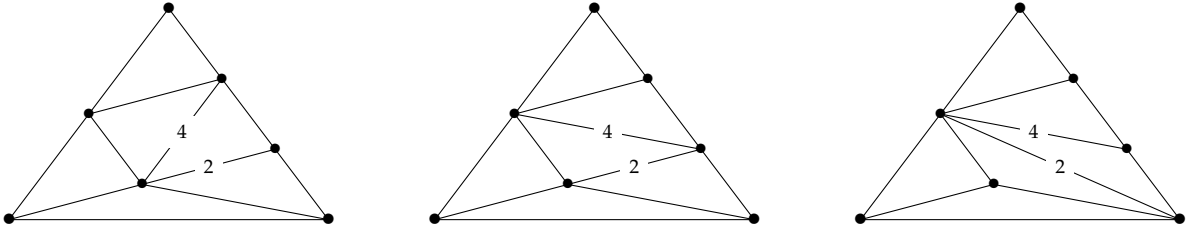
\begin{figure}[h]
\begin{center}
\begin{tikzpicture}[scale=0.7]
\node(e1) at (0+8,2){$\bullet$};
\node(e2) at (3+8,-2){$\bullet$};
\node(e3) at (-3+8,-2){$\bullet$};

{\small

\node(123) at (-0.5+8,-4/3){$\bullet$};
\node(240) at (2+8,2-8/3){$\bullet$};
\node(420) at (1+8,2-4/3){$\bullet$};
\node(303) at (-1.5+8,0){$\bullet$};

\draw[-] (e2.center) to (123.center);
\draw[-] (e3.center) to (123.center) to node[fill=white] {\tiny $2$} (240.center);
\draw[-] (303.center) to (420.center);
\draw[-] (123.center) to (303.center);
\draw[-] (123.center) to node[fill=white] {\tiny $4$} (420.center);

\draw[-] (e1.center) to (e2.center) to (e3.center) to (e1.center);
}
\node(e1) at (0+16,2){$\bullet$};
\node(e2) at (3+16,-2){$\bullet$};
\node(e3) at (-3+16,-2){$\bullet$};

{\small

\node(123) at (-0.5+16,-4/3){$\bullet$};
\node(240) at (2+16,2-8/3){$\bullet$};
\node(420) at (1+16,2-4/3){$\bullet$};
\node(303) at (-1.5+16,0){$\bullet$};

\draw[-] (e2.center) to (123.center);
\draw[-] (e3.center) to (123.center) to node[fill=white] {\tiny $2$} (240.center);
\draw[-] (303.center) to (420.center);
\draw[-] (123.center) to (303.center);
\draw[-] (303.center) to node[fill=white] {\tiny $4$} (240.center);

\draw[-] (e1.center) to (e2.center) to (e3.center) to (e1.center);
}

\node(e1) at (0+24,2){$\bullet$};
\node(e2) at (3+24,-2){$\bullet$};
\node(e3) at (-3+24,-2){$\bullet$};

\small

\node(123) at (-0.5+24,-4/3){$\bullet$};
\node(240) at (2+24,2-8/3){$\bullet$};
\node(420) at (1+24,2-4/3){$\bullet$};
\node(303) at (-1.5+24,0){$\bullet$};

\draw[-] (e2.center) to (123.center);
\draw[-] (e3.center) to (123.center);
\draw[-] (303.center) to node[fill=white] {\tiny $2$} (e2.center);
\draw[-] (303.center) to (420.center);
\draw[-] (123.center) to (303.center);
\draw[-] (303.center) to node[fill=white] {\tiny $4$} (240.center);

\draw[-] (e1.center) to (e2.center) to (e3.center) to (e1.center);
\end{tikzpicture}
\end{center}
\caption{Wall-crossings for $G=\frac{1}{6}(1,2,3)$}
\label{fig:1/6b}
\end{figure}

\end{example}

\section{Iterated Hilbert schemes} \label{sec:it_hilb}

Suppose $G\subseteq\on{SL}_3(\C)$ is a finite subgroup with a normal subgroup $A\lhd G$ and quotient $T=G/A$. $T$ acts on $A\hilb\C^3$ producing a crepant resolution
$$T\hilb A\hilb\C^3\to\C^3/G$$
One can easily extend this construction to longer chains of normal subgroups $A_1\lhd\dots\lhd A_s\lhd G$. We call the crepant resolutions constructed by this procedure \emph{iterated Hilbert schemes} (referred to as `Hilb of Hilb' in \cite{iin_gnh_13}). An explicit stability condition $\vartheta$ was constructed in \emph{ibid.}~to express $T\hilb A\hilb\C^3$ as a moduli space of quiver representations.

Suppose $\theta_A$ and $\theta_T$ are stability conditions for $Q_A$ and $Q_T$ respectively. As above it will be convenient to view $\theta_A$ (resp.~$\theta_T$) as a map from $\on{Irr}(A)$ (resp.~$\on{Irr}(T)$) to $\Q$, or more generally from the representation ring of $A$ (resp.~$T$) to $\Q$. Suppose further that $\theta_A$ and $\theta_T$ are positive on all nontrivial irreducible representations of $A$ and $T$ respectively -- we call such stability conditions \emph{zero-generated}. A stability condition $\vartheta\in\Theta(Q_G,\ub{d})$ producing $T\hilb A\hilb\C^3$ is defined by
$$\vartheta(\rho)=\begin{cases}
\theta_A(\rho|_A) & \rho\notin\on{Irr}(T) \\
\theta_A(\rho|_A)+\eps\cdot\theta_T(\rho|_T) & \rho\in\on{Irr}(T)
\end{cases}$$
for sufficiently small $\eps>0$. By $\rho\in\on{Irr}(T)$ we mean that $\rho$ is an irreducible representation of $G$ that is lifted from $T$ via the map $G\to T$; that is, $\rho|_A$ is trivial.

In the case that $G$ is abelian, Ishii--Ito--Nolla \cite[\S4.1]{iin_gnh_13} constructed a triangulation of the junior simplex giving $T\hilb A\hilb\C^3$ as a toric variety.

\begin{example} \label{ex:16_it} Consider $G=\frac{1}{6}(1,2,3)$. We consider the normal subgroup $A\cong\Z/2$ inside $G$ with quotient $T\cong\Z/3$. We show the triangulation for $T\hilb A\hilb\C^3$ on the left of Fig.~\ref{fig:1/6c}. Note that this is the second flop of $G\hilb\C^3$ considered in Ex.~\ref{ex:16_wall}. One can similarly start with the normal subgroup $A'\cong\Z/3$ with quotient $T'\cong\Z/2$. The iterated Hilbert scheme $T'\hilb A'\hilb\C^3$ is shown on the right of Fig.~\ref{fig:1/6c}. This is obtained by flopping the curve marked by $3$ in $G\hilb\C^3$.

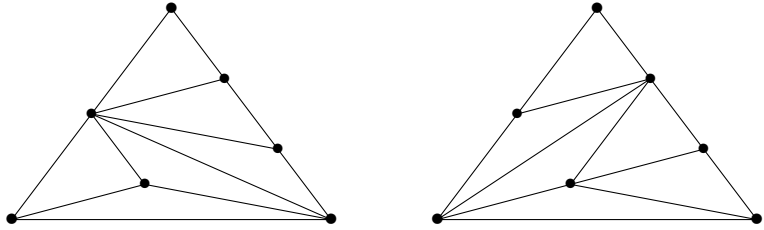
\begin{figure}[h]
\begin{center}
\begin{tikzpicture}[scale=0.7]
\node(e1) at (0+24,2){$\bullet$};
\node(e2) at (3+24,-2){$\bullet$};
\node(e3) at (-3+24,-2){$\bullet$};

{\small

\node(123) at (-0.5+24,-4/3){$\bullet$};
\node(240) at (2+24,2-8/3){$\bullet$};
\node(420) at (1+24,2-4/3){$\bullet$};
\node(303) at (-1.5+24,0){$\bullet$};

\draw[-] (e2.center) to (123.center);
\draw[-] (e3.center) to (123.center);
\draw[-] (303.center) to (e2.center);
\draw[-] (303.center) to (420.center);
\draw[-] (123.center) to (303.center);
\draw[-] (303.center) to (240.center);
}
\draw[-] (e1.center) to (e2.center) to (e3.center) to (e1.center);

\node(e1) at (0+32,2){$\bullet$};
\node(e2) at (3+32,-2){$\bullet$};
\node(e3) at (-3+32,-2){$\bullet$};

{\small

\node(123) at (-0.5+32,-4/3){$\bullet$};
\node(240) at (2+32,2-8/3){$\bullet$};
\node(420) at (1+32,2-4/3){$\bullet$};
\node(303) at (-1.5+32,0){$\bullet$};

\draw[-] (e2.center) to (123.center);
\draw[-] (e3.center) to (123.center);
\draw[-] (123.center) to (420.center);
\draw[-] (e3.center) to (420.center);
\draw[-] (123.center) to (240.center);
\draw[-] (303.center) to (420.center);
}
\draw[-] (e1.center) to (e2.center) to (e3.center) to (e1.center);
\end{tikzpicture}
\end{center}
\caption{Iterated Hilbert schemes for $G=\frac{1}{6}(1,2,3)$}
\label{fig:1/6c}
\end{figure}
\end{example}

\section{Connecting $G$-Hilb and iterated Hilbert schemes}

\subsection{Main conjecture} \label{sec:conj}

We start with the following lemma.

\begin{lemma} \label{lem:rep} With $G,A,T,\vartheta$ as above we have $\vartheta(\rho)<0$ if and only if $\rho\in\on{Irr}(T)$.
\end{lemma}

For convenience we denote the set of nontrivial irreducible representations of a group $H$ by $\on{Irr}^\star(H)$. We also write $\rho^A$ for the $A$-invariant part of a $G$-representation $\rho$, and $\rho_0$ for the trivial representation.

\begin{proof} It is irrelevant which zero-generated stability condition $\theta_A$ we choose for $A$ and $T$ and so we use
$$\theta_A(\rho)=\on{dim}{\rho}\text{ for $\rho\in\on{Irr}^\star(G)$}$$
We have
$$0=\theta_A(\C[A])=\theta_A(\rho_0)+\theta_A(\bigoplus_{\rho\in\on{Irr}^\star{G}}\on{dim}{\rho}\cdot\rho)=\theta_A(\rho_0)+\sum_{\rho\in\on{Irr}^\star(G)}(\on{dim}{\rho})^2=\theta_A(\rho_0)+|A|-1$$
and so $\theta_A(\rho_0)=1-|A|$. This stability condition has the benefit that $\theta_A(\psi)=\on{dim}{\psi}$ for any representation $\psi$ of $A$ with no trivial summand. It is clear that choosing $\eps$ sufficiently small causes $\vartheta$ to be negative on characters lifted from $T$ since $\theta_A(\rho|_A)<0$ for such $\rho$. Thus, suppose $\rho\in\on{Irr}^\star{G}$ is not lifted from $T$ and yet $\vartheta(\rho)<0$. Then
\begin{equation} \label{eqn:rep} \tag{$*$}
0>\theta_A(\rho|_A)=\theta_A(\rho^A)+\theta_A(\rho|_A/\rho^A)=\on{dim}{\rho^A}\cdot(1-|A|)+\on{dim}{\rho/\rho^A}=\on{dim}{\rho}-\on{dim}{\rho^A}\cdot |A|
\end{equation}
since $\rho^A$ is a trivial $A$-module. Note that $\rho^A$ is also a $G$-module hence, since $\rho$ is irreducible, either $\rho^A=\rho$ in which case $\rho\in\on{Irr}(T)$, or $\rho^A=0$ in which case
$$\on{dim}{\rho}>0=|A|\cdot\on{dim}{\rho^A}$$
and so the inequality (\ref{eqn:rep}) cannot be satisfied.
\end{proof}

This lemma states that the representations where the sign of $\vartheta$ differs from a stability condition in $\Theta^+(Q_G,\ub{d})$ defining $G\hilb\C^3$ are exactly those lifted from $T$. We take this much further in the following conjecture.

\begin{conjecture} \label{conj:to_it} Let $G\subseteq\on{SL}_3(\C)$ be abelian, and let $A\lhd G$ be a normal subgroup with quotient $T=G/A$. There is a path $\gamma\colon[0,1]\to\Theta(Q_G,\ub{d})$ passing through walls $\mfk{w}_1,\dots,\mfk{w}_m$ that satisfies the conditions in \S\ref{sec:wall} such that:
\begin{enumerate}
\item $\gamma(0)\in\mfk{C}_0$
\item $\M_{\gamma(1)}(Q_G,\ub{d})\cong T\hilb A\hilb\C^3$
\item All nontrivial irreducible representations lifted from $T$ are contained in $\bigcup_{i=1}^m\chi(\mfk{w}_i)$.
\end{enumerate}
\end{conjecture}

We are curious if anything stronger is true especially when $T$ is small, where it may be possible that there is a path from $G\hilb\C^3$ to $T\hilb A\hilb\C^3$ for which `most' of the walls it crosses are labelled with a lifted character from $T$. One can formulate a version of this conjecture for longer chains of normal subgroups and the corresponding iterated Hilbert schemes.

\subsection{Partial results and remarks} \label{sec:ex} We conclude this section with some glimpses of how one might approach Conjecture \ref{conj:to_it}.

\begin{prop} \label{prop:partial} Let $G/A=T$ as above. Suppose edges $l_1,\dots,l_p$ labelled with representations $\rho_1,\dots,\rho_p$ in the triangulation for $G\hilb\C^3$ are not in the triangulation for $T\hilb A\hilb\C^3$. Then there is a path $\gamma$ satisfying the conditions of Conjecture \ref{conj:to_it} with $\{\rho_1,\dots,\rho_p\}\subseteq\chi(\gamma)$.
\end{prop}

\begin{proof} We note that the curves corresponding to the edges $l_i$ must be flopped (possibly multiple times) in order to obtain the triangulation for $T\hilb A\hilb\C^3$ from the triangulation for $G\hilb\C^3$. We thus apply the methods of \cite[\S8]{ci_flo_04} to find a path $\gamma$ passing through walls $\mfk{w}_1,\dots,\mfk{w}_q$ such that:
\begin{itemize}
\item for each $j=1,\dots,p$ there is $i_j\in\{1,\dots,q\}$ such that crossing $\mfk{w}_{i_j}$ realises a flop in the curve corresponding to the edge $l_j$,
\item every other wall that $\gamma$ passes through is type \texttt{0}, \texttt{0'}, or \texttt{I}.
\end{itemize}
By construction $\chi(\mfk{w}_{i_j})=\{\rho_j\}$, which gives the result.
\end{proof}

\begin{cor} Continuing the notation of Prop.~\ref{prop:partial}, if all the representations lifted from $T$ label edges in the triangulation for $G\hilb\C^3$ that are not in the triangulation for $T\hilb A\hilb\C^3$, then Conjecture \ref{conj:to_it} is true for $G$.
\end{cor}

From Ex.~\ref{ex:16_wall} and Ex.~\ref{ex:16_it} this applies to both decompositions of $\frac{1}{6}(1,2,3)$ and to several other examples, such as when $A$ is any subgroup of $G=\frac{1}{30}(2,3,25)$ except $\frac{1}{2}(0,1,1)$. We note however that there are many examples where representations lifted from a quotient label divisors, such as for $G=\frac{1}{25}(1,3,21)$ and $A=\frac{1}{5}(1,3,1)$ as can be seen from \cite[Fig.~29]{wor_wal_20}.

We next interpret and verify Conjecture \ref{conj:to_it} for several polyhedral subgroups. A finite subgroup of $\on{SL}_3(\C)$ is \emph{polyhedral} if it is conjugate to a subgroup of $\on{SO}(3)$. The conjugacy classes of these subgroups are classified by the ADE Dynkin diagrams; the groups themselves are discussed in \cite[\S2]{ns_flo_17}.

\begin{example} We first consider the subgroup $G$ of type $D_{2n}$ with its index $2$ subgroup $A\cong\Z/n$. As usual, denote the quotient $G/A=T\cong\Z/2$. We can compute the stability condition $\vartheta$ from \cite[Def.~2.4]{iin_gnh_13} giving $T\hilb A\hilb\C^3$. We see, as expected from Lemma \ref{lem:rep} that $\vartheta$ is only negative on the two one dimensional representations of $G$ that are lifted from $T$. Write $L$ for the nontrivial such representation. We use the explicit calculation of chambers in \cite[Thm.~6.4(i)-(ii)]{ns_flo_17}.

When $n$ is odd we see that $\vartheta$ lives in the chamber adjacent to $\mfk{C}_0$ found by crossing the wall $\mfk{w}$ with unstable locus the curve $C$ corresponding to $L$. By the heuristic of \S\ref{sec:wall} it is very reasonable to label the wall with the representation $L$, and hence immediately conclude that this version of Conjecture \ref{conj:to_it} holds for such subgroups.

In the case when $n$ is even a similar calculation can be implemented to show that $\vartheta$ resides in a chamber adjacent to $\mfk{C}_0$ found by crossing a wall that realises the flop in the curve corresponding to the nontrivial irreducible representation lifted from $T$, hence giving an appropriately modified version of Conjecture \ref{conj:to_it} as for odd $n$.

The last case we treat, following \cite{ns_flo_17}, is the tetrahedral group $G\subseteq\on{SO}(3)$ of order $12$. This has a normal subgroup $A=\frac{1}{2}(1,1,0)\times\frac{1}{2}(1,0,1)\cong\Z/2\times\Z/2$ with quotient $T\cong\Z/3$. The three $1$-dimensional representations of $G$ are those lifted from $T$ and so the stability condition $\vartheta$ describing $T\hilb A\hilb\C^3$ is negative on them. \cite[Thm.~6.4(iii)]{ns_flo_17} describes the chamber structure in this case. We see from this that the resolution $X_{12}$ (in Nolla--Sekiya's notation) obtained by flopping the curves corresponding to the two nontrivial representations lifted from $T$ in any order is isomorphic to $T\hilb A\hilb\C^3$. Thus again we have a suitable version of Conjecture \ref{conj:to_it}.
\end{example}

\begin{remark} A large family of examples that would be interesting to study are the trihedral groups \cite{ito_cre_95,len_mck_02,wor_non_15} of which the tetrahedral group is the smallest. These are groups of the form $A\rtimes T$ where $T\cong\Z/3$ is generated by
$$\tau=\mat{ccc}
0 & 1 & 0 \\
0 & 0 & 1 \\
1 & 0 & 0\tam$$
Recent, as-yet-unpublished work of Nolla constructs a version of Reid's recipe for the next smallest trihedral group $\frac{1}{7}(1,2,4)\rtimes T$, and Ito--Takahashi have built $T\hilb A\hilb\C^3$ for this example. A natural next step is to compare these through the lens of Conjecture \ref{conj:to_it}.
\end{remark}

\begin{remark} \label{rem:mutation} One of the tools that supports extra insight in the polyhedral case is a bijection between crepant resolutions of $\C^3/G$ and mutations of a `quiver with potential' $(Q,W)$ -- see \cite{dwz_qui_08,dwz_qui_10} for much of the background -- such that mutation in a vertex corresponds to flopping a corresponding curve. This quiver with potential is constructed from the representation theory of $G$; indeed, $Q=Q_G$ is just the McKay quiver.

It is natural to conjecture that there is a quiver with potential $(Q,W)$ whose mutations similarly biject with the crepant resolutions of $\C^3/G$ for abelian (or even all) subgroups $G\subseteq\on{SL}_3(\C)$. In the abelian case it is equivalent to find a quiver with potential whose mutations are in bijection with unimodular triangulations of the junior simplex, and for which mutation in a vertex corresponds to flipping a corresponding edge. One might compare the various works in the mainline theory of cluster algebras constructing quivers (sometimes with potential) to capture triangulations of marked surfaces, for instance \cite{lab_tri_16,fst_clu_08,ft_clu_08}.
\end{remark}

\begin{remark} Both Conjecture \ref{conj:to_it} and the conjecture of Rem.~\ref{rem:mutation} have analogues for more general Gorenstein toric singularities via the theory of \emph{dimer models} \cite{bcv_geo_15,iu_dim_16}, which may even be a more amenable setting in which to study them; especially with the recent developments in Reid's recipe for dimer models \cite{cht_com_20}.
\end{remark}

\subsection*{Acknowledgements} The author is grateful for helpful conversations with Alastair Craw, Tom Ducat, \'Alvaro Nolla de Celis, Yukari Ito, Jonathan Lai, and Tim Magee, and to the organisers of the conference `The McKay Correspondence, Mutations, and Related Topics' hosted remotely by Kavli IPMU in July-August 2020 for providing an excellent opportunity -- especially given the backdrop of the \textsc{covid-19} pandemic -- to discuss and share some of these ideas.

\bibliographystyle{acm}
\bibliography{bw}

\end{document}